\documentclass[11pt, a4paper]{article}

\usepackage[english]{babel}
\usepackage[utf8x]{inputenc}
\usepackage[T1]{fontenc}
\usepackage{xparse,amsmath}
\usepackage{amsfonts}
\usepackage{amsthm}
\usepackage{amssymb}
\usepackage{graphicx}
\usepackage[shortlabels]{enumitem}

\usepackage[a4paper,top=3cm,bottom=2cm,left=3cm,right=3cm,marginparwidth=1.75cm]{geometry}

\usepackage{amsfonts}
\usepackage[colorinlistoftodos]{todonotes}
\usepackage{hyperref}

\hypersetup{
  colorlinks   = true,    
  urlcolor     = blue,    
  linkcolor    = cyan,    
  citecolor    = red      
}

\title{On the resolution of the sensitivity conjecture}
\author{Rohan Karthikeyan 
\qquad Siddharth Sinha \qquad  Vallabh Patil\\
\small Department of Applied Mathematics\\
\small Delhi Technological University\\ 
\small Delhi, India\\
\small\tt fromrohank07@gmail.com\\}

\begin{document}
\theoremstyle{plain}
\newtheorem{theorem}{Theorem}[section]
\newtheorem{lemma}[theorem]{Lemma}
\newtheorem{corollary}[theorem]{Corollary}
\newtheorem{proposition}[theorem]{Proposition}
\newtheorem{conjecture}[theorem]{Conjecture}
\newtheorem{criterion}[theorem]{Criterion}
\newtheorem{algorithm}[theorem]{Algorithm}
\newtheorem{condition}[theorem]{Condition}
\newtheorem{problem}[theorem]{Problem}
\newtheorem{example}[theorem]{Example}
\newtheorem{exercise}{Exercise}[section]
\newtheorem{obs}{Observation}
\newtheorem{note}[theorem]{Note}
\newtheorem{notation}[theorem]{Notation}
\newtheorem{claim}[theorem]{Claim}
\newtheorem{summary}[theorem]{Summary}
\newtheorem{acknowledgment}[theorem]{Acknowledgment}
\newtheorem{case[theorem]}{Case}
\newtheorem{conclusion}[theorem]{Conclusion}
\newtheorem{definition}[theorem]{Definition}
\newtheorem{remark}[theorem]{Remark}

\def\bt{\mathbf t}
\def\bZ{\mathbb{Z}}
\def\bS{\mathbb{R}}
\def\bN{\mathbb{N}}

\maketitle

\begin{abstract}
The Sensitivity Conjecture is a long-standing problem in theoretical computer science that seeks to fit the sensitivity of a Boolean function into a unified framework formed by the other complexity measures of Boolean functions, such as block sensitivity and certificate complexity. After more than thirty years of attacks on this Conjecture, Hao Huang (2019) gave a very succinct proof of the Conjecture. In this survey, we explore the ideas that inspired the proof of this Conjecture by an exposition of four papers that had the most impact on the Conjecture. We also discuss progress on further research directions that the Conjecture leads us to.

\end{abstract}

\section{Introduction}

Boolean functions map a sequence of bits to a single bit vector $0$ or $1$, represented as False and True respectively. Some of the simplest Boolean functions are the AND function $x \cdot y$, the OR function (non-exclusive) $x + y$, and the NOT function $\bar{x} = 1- x$. While analysing Boolean functions, we would naturally want to analyse the complexity of these functions. Precisely, this problem can be stated as follows (see \cite{stasys}) - 
\begin{center}
    How many of these simplest functions do we need to use to calculate the value of a given Boolean function on all the input vectors?
\end{center}

There are several useful measures to describe the complexity of a Boolean function that can be stated mathematically. Two such measures are sensitivity and block sensitivity. The \textit{sensitivity} $s(f)$ of a Boolean function $f:\{0, 1\}^n \rightarrow \{0, 1\}$ on $n$ variables is defined to be the largest number for which there is an $x$ such that there are at least $s(f)$ values of $i = 1, \cdots, n$ with $f(x + e_i) \neq f(x)$. Here $e_i$ is the Boolean vector with exactly one $1$ in the $i$-th position. Further, the \textit{block sensitivity}, $bs(f)$ is defined to be the maximum number of disjoint subsets of $B_1, \cdots, B_t$ of $[n] = \{1, \cdots, n\}$ such that for all $1 \leq j \leq t, f(x) \neq f(x^{B_j})$ where $x^{B_j}$ is the Boolean string which differs from $x$ on exactly the bits of $B_j$. \newline

It may be useful to think of sensitivity of a Boolean function in the following way: Suppose there is an array of $n$ switches, with some wiring for a single light bulb. For different configurations of switches flipped on or off, the light bulb is either on, or off. Then the circuit is said to be sensitive, or \textit{critical}, with respect to the $i$-th switch if for some configuration of the states of the switches, flipping the $i$-th switch  will change the state of the light bulb from on to off, or vice versa. For that configuration, there may be more than one such switch for which the circuit is sensitive. If we count the number of such switches at which the circuit is sensitive, and do so for every configuration of the switches, then the greatest number of switches for which the circuit was sensitive for some configuration, is its global sensitivity, or its \textit{critical complexity}.\newline 

Now, block sensitivity is known to be polynomially related to a number of other complexity measures of $f$, including the decision-tree complexity, and the certificate complexity. A long-standing open question that existed was whether sensitivity also belonged to this equivalence class. A positive answer to this question is commonly known as the \textit{Sensitivity Conjecture} first proposed by Nisan and Szegedy \cite{Nisz}.\newline

The Conjecture can also be seen as concerning whether sensitivity is polynomially related to $n$, the number of input variables to the Boolean function. To resolve this, Gotsman and Linial \cite{Gotsman} proved the equivalence between the following two problems:
\begin{enumerate}
    \item Denote the $n$-dimensional cube by $Q_n= \{-1, 1\}^n$ and the maximal degree of any graph $G$ by $\Delta(G)$. Now, for an induced subgraph $G$ of $Q_{n}$ with strictly greater than half the number of vertices in $Q_n$, i.e., greater than $2^{n-1}$ vertices, can $\Delta(G)$ be bounded below by a function of $n$?
    \item Let $f: Q_n \rightarrow \{-1, 1\}$ be a Boolean function, with its sensitivity as defined above. Denoting the degree of the multilinear polynomial representation of $f(x)$ by $\deg(f)$, is $\deg(f)$ bounded above by a function of $s(f)$?
\end{enumerate}

The equivalence established by Gotsman and Linial was predicated of the representation of a Boolean function by the unique 2-colorings of the vertices of the $n$-dimensional hypercube. The two possible colors correspond to either $-1$ or $1$. It was previously established by Chung et. al. \cite{Chung} that for an induced sub-graph $G$ of $Q_n$ with strictly more than half the vertices of the $Q_n$, and for some vertex $v \in {G}$, the degree of $v$ in $G$ is bounded above logarithmically in $n$. Recently, this bound was improved from the logarithmic bound to a polynomial relation \cite{Hao}.
\begin{theorem} [Hao]
For every integer $n\geq 1$, let G be an arbitrary $(2^{n-1}+1)$-vertex induced subgraph of $Q_n$, then $$\Delta(G)\geq \sqrt{n}$$
\end{theorem}

This result directly resolves the Sensitivity Conjecture:
\begin{conjecture}[Sensitivity Conjecture]
There exists an absolute constant $C > 0$, such that for every Boolean function $f$, $$bs(f)\leq s(f)^{C}$$
\end{conjecture}


In the past three decades, a great amount of research has been done on resolving the Sensitivity Conjecture, resulting in a wide-ranging body of knowledge. In this survey, we could not hope to cover even a fraction of it. The selection of material is largely limited to results that have had the most impact on the development of the Conjecture. We would like to direct the attention of the interested reader to the surveys of Buhrman and de Wolf \cite{wolf}, and  Hatami, Kulkarni and Pankratov \cite{haka} for more in-depth discussions on the Conjecture.\newline

In this survey, we will expound on the progression of results by Chung, Gotsman and Linial, Nisan, and ultimately, Hao, and the connections between their results. We first delve into the complexity measures of Boolean functions in section \hyperref[sec:Nota]{2}. In section \hyperref[sec:Chung]{3}, the combinatorial proof of the logarithmic bound established by Chung et al., will be analysed with special consideration to its connections with the geometry of the hypercube. We will also be looking at Nisan's results characterising the degree of Boolean functions in terms of combinatorial properties of those functions in section \hyperref[sec:Nisan]{4}. 
Furthermore, in section \hyperref[sec:GL]{5}, we will discuss Gotsman and Linial's proof to the equivalence between the problem of the order of hypercubes and the degrees of their induced subgraphs, and the relation between the degree and sensitivity of Boolean functions. In section \hyperref[sec:Hao]{6}, we will discuss Hao's succinct proof and then move on to providing comments on some open problems of Hao in section \hyperref[sec:Fur]{7}, before finishing with the relation of the Conjecture to CREW PRAMs.\newline

We would like to thank our project advisor, Dr. Chandra Prakash Singh for his constant encouragement in the preparation of this manuscript. We would also like to thank the Quanta magazine \cite{gil} for introducing us to this wonderful topic of sensitivity. Last but not the least, we would like to thank our parents for providing constant moral support.

\section{Some Complexity Measures of Boolean Functions}
\label{sec:Nota}
Consider the set $Q_n = \{-1, 1\}^n$. The \textit{hypercube} of dimension $n$ is the graph with vertex set $Q_n$ and an edge between $x =(x_1, x_2, \cdots, x_n)$ and $y = (y_1, y_2, \cdots, y_n)$ in $Q_n$ if there is exactly one $i$ such that $x_i \neq y_i$. Let $f$ be a function mapping $Q_n$ to $\{-1, 1\}$. 

\begin{definition} The \textit{sensitivity} of $f$ at $x$, denoted $s(f, x)$, is defined as the number of neighbours $y$ of $x$ for which $f(x) \neq f(y)$. 
\end{definition}
In other words, it's the number of $i$ such that: $$ f(x_1, \cdots, x_i, \cdots, x_n) \neq f(x_1, \cdots, -x_i, \cdots, x_n)$$ 
\begin{definition}
The \textit{(overall) sensitivity} of $f$, also called the critical complexity, denoted $s(f)$, is the maximum over all $x$ in $Q_n$ of the sensitivity of $f$ on $x$, that is, $$ s(f) = \max_{x\, \in\, Q_n} s(f, x)$$
\end{definition}

We now move on to the second complexity measure. For a string $x \in \{0, 1\}^n$ and a set $S \subseteq [n]: = \{1, 2, \cdots, n\}$, we define $x^{(S)}$ to be the Boolean string which differs from $x$ on exactly the bits in $S$. 

\begin{definition} The \textit{block sensitivity} of $f$, denoted by $bs(f)$, is the maximum number $t$ such that there exists an input $x \in \{0, 1\}^n$ and $t$ disjoint subsets $B_1, \cdots, B_t \subset [n]$ such that for all $1 \leq i \leq t$, $f(x) \neq f(x^{(B_i)})$.
\end{definition}

Before proceeding further, it will be worthwhile to have a look into the decision tree model. In this model, we would like to compute the value of a given function at an unknown input. To do so, we collect information about the input by asking questions. Considering Boolean functions, we will ask only binary queries, that is, the possible answers to the query will be $0$ and $1$. Now, each question asked depends only on the information gained so far. 

A \textit{decision tree (deterministic)} can then be defined as a rooted binary tree with labels on each node and edge. Each inner node is labelled by a query. One of the two edges leaving the node is labeled $0$, the other is labelled $1$. The two labels represent the two possible answers to the query. Each leaf, labelled $0$ or $1$, give the output.\newline

The Boolean function computed by the given decision tree takes the label at this leaf as the value on the given input. Now, we define $\text{cost}(A, x)$ as the number of queries asked when the decision tree $A$ is executed on input $x$. That is, it's the length of the computation path forced by $x$. Also, $\max_{x} \text{cost}(A, x)$ defines the worst case complexity of $A$, that is, the tree's depth.

\begin{definition}
The \textit{decision tree complexity} of $f$, denoted by $D(f)$, is given by: $$D(f) = \min_{A} \max_{x} \text{cost}(A, x)$$ 
\end{definition}
That is, it's the depth of the minimum-depth decision tree that computes $f$. It is obvious that any function $f: \{0, 1\}^n \rightarrow \{0, 1\}$ can be computed by asking $n$ questions.\newline

And, moving towards the last complexity measure, define a $b$-certificate (for $b \in \{0, 1\}$) as a partial assignment $p: S \rightarrow \{0, 1\}$, which forces the value of $f$ to $b$. Here, $S$ is some subset of the concerned variables. 

\begin{definition}
The \textit{certificate complexity} of a Boolean function $f$ on $x$, denoted $C_{x}(f)$, is the size of the smallest certificate that agrees with $x$. 
\end{definition}
Now, the certificate complexity of $f$, denoted $C(f)$, is defined as $C(f) = \max_{x \, \in \,Q_n} C_{x}(f)$. How can we visualise this measure? We can think of a Boolean function as coloring the vertex on a $n$-dimensional hypercube using just two colors. Then, the certificate complexity is $n$ minus the dimension of the largest monochromatic hypercube in $Q_n$.

\section{Chung's result}
\label{sec:Chung}
\textbf{\underline{Aim:}} In this section, we expound Chung's result \cite{Chung}: An induced subgraph of the $n$-dimensional cube graph $Q_n$ with strictly more than half of its vertices has a maximum degree of at least $o(\log n)$. Also, they constructed a $(2^{n-1} + 1)$-vertex induced subgraph whose maximum degree is $\lceil \sqrt{n} \, \rceil$.\newline

First, we need two definitions.
\begin{definition}
Take a graph $G = (V, E)$. For $S \subset V$, an \textit{induced subgraph} $G[S]$ is the graph with vertex set $S$ and edge set $E'$ where $E'$ consists of all edges in $E$ that have both endpoints in $S$.
\end{definition}

\begin{definition}
The \textit{maximum degree} $\Delta(G)$ of a graph $G = (V, E)$ is defined as: $$\Delta(G) = \max_{v \in V(G)} \deg_G(v)$$
\end{definition}

Now, the theorem we want to prove is:
\begin{theorem} Let $G$ be an induced subgraph of $Q_n$ with at least $2^{n-1} + 1$ vertices. Then for some vertex $v$ of $G$, we have: \begin{equation} \deg_G(v) > \frac{1}{2}\log n - \frac{1}{2}\log \log n + \frac{1}{2} \label{degGv} \end{equation} Also, there exists a $(2^{n-1}+1)$ vertex induced subgraph G of $Q_n$ with \begin{equation} \Delta(G) < \sqrt{n} + 1 \end{equation}
\end{theorem}

\subsection{Proving the Upper bound}
Let's first denote $\{1, 2, \cdots, n\} := [n]$. To prove the upper bound, we consider $[n]$ and a family of sets $2^{[n]}$, that is, the set of all subsets of $[n]$. We observe that the natural bijection between the set of $n$-bit vectors $\{0, 1\}^n$ and $2^{[n]}$ is the map $\psi: \{0, 1\}^n \leftrightarrow 2^{[n]}$ such that the vector $\mathbf{x} = (x_1, x_2, \cdots, x_n)$ is mapped to $T = \{i \mid x_i = 1, 1 \leq i \leq n\}$ and vice versa. In other words, the $i$th bit of $\mathbf{x}$ denotes the inclusion of the $i$th natural number in the subset of $[n]$; if $x_i = 1$, then $T$ includes $i$.\newline

Therefore, we may extend the bijection between the $n$-bit vectors and $2^{[n]}$ to the construction of $Q_n$. We partition $Q_n$ into two vertex-disjoint subgraphs $G$ and $G'$ where $G$ has $2^{n-1} + 1$ vertices. We consider now a certain family of sets $\mathbf{X}$ such that $|\textbf{X}| = 2^{n-1} \pm 1$, where $\textbf{X} \subset 2^{[n]}$. We define the family of sets \textbf{X} over a finite collection of subsets $\textbf{F} \subset 2^{[n]}$ as: \begin{align*} \textbf{X}(\textbf{F}) = &\{S \subset [n] \, :  |S| = even, \exists \, F \in \textbf{F} \, \text{with} \,  F \subset S\} \\ &\cup \, \{S \subset [n] \, : |S| = odd, F \setminus S \neq \phi \, \text{for all} \,  F \in \textbf{F}\} \end{align*} That is, \textbf{X}(\textbf{F}) is the collection of all even sets which contain some $F \in \textbf{F}$ along with all odd sets which contain no $F \in \textbf{F}$.\newline

Let us consider \textbf{F} as the partition of $[n]$ given by $[n] = F_1 \cup F_2 \cdots \cup F_k$ such that $|k - \sqrt{n}| < 1$ and $||F_i| - \sqrt{n}| < 1$, $1 \leq i \leq k$. We now prove the existence of such a partition:
\begin{proof}
Let $k = \lceil \sqrt{n} \, \rceil$, and this implies that $|k - \sqrt{n}| < 1$. Now, denote $u: = \sqrt{n} \implies u^2 = n$. It's easy to see that there exists an $\epsilon$ with $0 \leq \epsilon < 1$ so that $\lfloor \sqrt{n} \rfloor = \sqrt{n} - \epsilon = u - \epsilon $ and $\lceil \sqrt{n} \, \rceil = \sqrt{n} - \epsilon + 1 = u - \epsilon + 1$.\newline
 
For the above partition to exist, we simply need to show that $\exists \, x, y \in \bZ^+$ such that $x \lfloor \sqrt{n} \rfloor + y \lceil \sqrt{n} \, \rceil = n$ with $x + y =  \lceil \sqrt{n} \, \rceil = u - \epsilon + 1$. Now, note that \begin{equation} x(u - \epsilon) + y(u - \epsilon + 1) = u^2 \end{equation} must have integer solutions. Observe that: \begin{align*} (3) &\implies x(x + y-1) + y(x + y) = u^2 \\ &\implies x^2 + y^2 + 2xy - x = u^2 \\ &\implies (x + y)^2 = u^2 + x \end{align*} Clearly, this has integer solutions as we may fix $u^2$ and $x$ and so vary $y$ to get a set of solutions.
\end{proof}

For such a partition of $[n]$, we wish to see if the following claims are true: \begin{enumerate}
    \item $|\mathbf{X}| = 2^{n-1} \pm 1$.
    \item For the subgraph induced by $\textbf{X}(\textbf{F})$ and $2^{[n]} \setminus \textbf{X}(\textbf{F})$, $\Delta \leq k$.
\end{enumerate}

Let the reader allow us to introduce some terminology here. The \textit{rank} of \textbf{F}, denoted $r(\textbf{F})$, is the largest size of an element in \textbf{F}. We denote $t(\textbf{F})$ to be the \textit{largest size of the disjointly representable subsystems} of \textbf{F}. That is, $t(\textbf{F})$ is the maximum value of $t$ such that one can find $F_1, F_2, \cdots, F_t \in \textbf{F}$ and $x_i \in F_i$, $1 \leq i \leq t$, such that $x_i \in F_j \leftrightarrow i = j$, for all $1 \leq i, j \leq k$, that is, no $F_i$ is contained in the union of the others.

In the partition chosen for \textbf{F} such that we had $k$ disjointly representable subsets of \textbf{F}, we note that $t(\textbf{F}) = k$ and $r(\textbf{F})$ is either $\lfloor \sqrt{n} \rfloor$ or $\lceil \sqrt{n} \, \rceil$ (How?). Since, the degree of a vertex corresponding to some $S \in \textbf{X}(\textbf{F})$ is the number of $S_i \in \textbf{X}(\textbf{F})$ such that $S$ and $S_i$ differ by only one element, the maximum degree $\Delta(G(\textbf{F}))$ is either bounded by $r(\textbf{F})$ or $t(\textbf{F})$. We can then easily see that $\Delta(G(\textbf{F})) < \sqrt{n}$.
\begin{lemma}
$\Delta(G(\textbf{F})) \leq \max\{r(\textbf{F}), t(\textbf{F})\}$
\end{lemma}
\begin{remark}
The same result also holds true for $\Delta(G'(\textbf{F})).$
\end{remark}
\begin{proof}
An interesting thing to note is that this result holds for any family of subsets of $2^{[n]}$. To see this, consider the sets corresponding to the vertices: $(1\, 1\, 0\, 1\, 1)$, $(1\, 1\, 0\, 1\, 0)$, $ (0\, 1\, 0\, 1\, 1)$ of $G(\textbf{F})$ for some $\textbf{F} \subset 2^{[n]}$. The subsets $\{1, 2, 4\}$ and $\{2, 4, 5\}$ belong to $\textbf{X}(\textbf{F})$ so long as no subset of either belongs to \textbf{F}. Therefore, there are no even subsets in \textbf{F} which are contained in odd sets of $\textbf{X}(\textbf{F})$.\newline

Consequently, $(1\, 1\, 1\, 1\, 1)$ cannot belong to $G(\textbf{F})$ as it would contain some $F \in \textbf{F}$. Consider again the vertex $(1\, 1\, 0\, 1\, 1)$. It can have maximum degree $4$ (in the $5$-cube graph) if $(0\, 1\, 0\, 1\, 1)$, $(1\, 0\, 0\, 1\, 1)$,  $(1\, 1\, 0\, 0\, 1)$ and $(1\, 1\, 0\, 1\, 0)$ belong to $G(\textbf{F})$. However, if no subsets of the vertices adjacent to $(1\, 1\, 0\, 1\, 1)$ belong to $\textbf{F}$, then only the singleton elements $(1\, 0\, 0\, 0\, 0)$, $(0\, 1\, 0\, 0\, 0), \cdots, (0\, 0\, 0\, 0\, 1)$ may belong to $G(\textbf{F})$, each with degree at most $1$.\newline

Therefore, for any edge of $Q_n$ corresponding to $(S, S')$, where $S, S' \in \textbf{X}(\textbf{F})$ and $S$ is even, the odd set $S'$ is seen to be a proper subset of $S$. Since, $S$ and $S'$ differ by the inclusion of precisely one element of $[n]$, and the exclusion of that element from $S$ results in a subset not contained in $\textbf{F}$, then \begin{equation} \deg(S) \leq \bigl\lvert \bigcap \{F: \, F \in \textbf{F}\, , F \subset S\} \bigr\rvert \leq r(\textbf{F}) \end{equation}

Taking $S$ as the odd set, we instead get $S \subset S'$, and by the construction of $G(\textbf{F})$, there exists a $F \in \textbf{F}$ such that $F \subset S'$ and $F \not\subset S$. Suppose, we have the sets $S_1', S_2', \cdots, S_a'$ such that there is an edge between $S$ and $S_i'$, $1 \leq i \leq a$. Then, there must also be some sets $F_1, F_2, \cdots, F_a \in \textbf{F}$ which don't belong to $S$, but $F_i \in S_i'$. Since, $S$ and $S_i'$ differ by the inclusion of exactly one element $F_1, F_2, \cdots, F_a$ must be disjointly representable. Hence, $a \leq t(\textbf{F})$. Since the quantity $a$ gives the number of adjacent vertices, we have completed the proof of the lemma.
\end{proof}

Now, the cardinality of $\textbf{X}(\textbf{F})$ can be found using the sieve formula and is given as: \begin{equation} |\textbf{X}(\textbf{F})| = 2^{n-1} + (1-2\epsilon)\bigg[\sum_{\substack{F_i \in \textbf{F} \\ |F_i| = n}} 1 - \underset{|F_i \, \cup \, F_j| = n}{\sum\limits_{F_i \in \textbf{F}}\sum\limits_{F_j \in \textbf{F}}} 1 + \cdots \bigg] \label{sieve} \end{equation}
where $\epsilon = 0$ when $n$ is even, and equals $1$ when odd.\newline

Denote by $f(\textbf{F})$ the bracketed expression on the RHS of (\ref{sieve}). When we take \textbf{F} as the partition of $[n]$ into $F_1, F_2, \cdots, F_k$, $f(\textbf{F})$ can be seen to equal $(-1)^{k+1}$. And, also we can verify that $(1 - 2\epsilon) = (-1)^n$ giving: $$|\textbf{X}(\textbf{F})| = 2^{n-1} + (-1)^{n + k + 1}$$ thus proving claim $1$.\newline

Now, to prove claim $2$, since we are concerned only with partitions where one part contains more than half of the vertices, we consider the case where $f(\textbf{F}) \neq 0$, and prove this lemma: 
\begin{lemma}
Suppose $f(\textbf{F}) \neq 0$. Then, $\max\{r(\textbf{F}), t(\textbf{F})\} \geq \sqrt{n}.$
\end{lemma}
\begin{remark}
We choose the maximum of $r(\textbf{F})$ and $t(\textbf{F})$ since the degree is bounded by both values, and the greater value is the one which dominates the bound.
\end{remark}
\begin{proof}
Observe that the motivation for $f(\textbf{F}) \neq 0$ stems from our choice of a partition in which one subgraph of the hypercube contains strictly greater than half the vertices. Taking strictly more than half of the vertices implies that there is a vertex in every direction. Take any $n$-bit vector. Now, there are $2^{n-1}$ possible $n$-bit vectors such that every one of them has their $j$th component as $0$. Since, we are taking $2^{n-1} + 1$ vertices, by the pigeon-hole argument, there is at least one vector in every direction, that is, has a non-zero value for the $j$th component. Since vertices of $G(\textbf{F})$ correspond to the subsets of $2^{[n]}$, this implies that $\bigl\lvert \bigcap \textbf{F} \bigr\rvert = \sqrt{n}$.\newline

Let us choose a subfamily $\{F_1, F_2, \cdots, F_s\}$ of $\textbf{F}$ such that this family has the smallest number of sets such that $\bigcup\limits_{i=1}^{s} F_i = [n]$. Now, $r(\textbf{F}) < \sqrt{n}$, else the inequality $\max\{r(\textbf{F}), t(\textbf{F})\} \geq \sqrt{n}$ is trivial for $|F| \geq \sqrt{n}$. Then, we must check if $t(\textbf{F}) \geq \sqrt{n}$ in such a case. Since, $s$ is the minimum number such that $\bigcup\limits_{i=1}^{s} F_i = [n]$, so $\{F_1, F_2, \cdots, F_s\}$ must be disjointly representable, else there would be a contradiction to the minimality of $S$.\newline

Since $|F_i|$ cannot exceed $\sqrt{n}$, we clearly see that if $\max |F_i| < \sqrt{n}$, and if $s < \sqrt{n}$, then: $s\cdot \max\{|F_i|\} < |\, [n] \, | = n$, which is a contradiction since $\bigcup\limits_{i=1}^{s} F_i = [n]$. Therefore $s \geq \sqrt{n}$, and so: $$\sqrt{n} \leq \max\{r(\textbf{F}), t(\textbf{F})\}$$

Since $\sqrt{n} < n$ for $n \geq 1$, the maximal degree $\Delta(G) \leq \sqrt{n}$, and since both $\Delta(G)$ and $\sqrt{n}$ are bounded above by $\max\{r(\textbf{F}), t(\textbf{F})\}$, we can have a partition $G(\textbf{F})$ such that $\Delta(G) \ll \sqrt{n}$.
\end{proof}

\subsection{Proving the Lower bound}
To show (\ref{degGv}), we first need a lemma:
\begin{lemma}
Let $G$ be a subgraph of $Q_n$ with average degree $\bar{d}$. Then, $|V(G)| \geq 2^{\bar{d}}$.
\end{lemma}
\begin{proof}
Let us use induction on $|V(G)|$. We can split $Q_n$ into two $(n-1)$-dimensional subcubes $Q_1$ and $Q_2$ such that $V_1 = Q_1 \cap V(G) \neq \phi$ and $V_2 = Q_2 \cap V(G) \neq \phi$. Also suppose that $|V_2| \geq |V_1|$ and that there are $s$ edges between $V_1$ and $V_2$ in $G$, which would imply that $|V_1| \geq s$, because no more than one edge can come out of any vertex in $V_1$. Let the restriction of $G$ to $V_i\, (i= 1, 2)$, be denoted by $G_i$. The induction hypothesis now gives us: $$|V_i|\log_2|V_i| \geq \sum \deg_{G_i}(v) = \sum_{v \, \in \, V_i} \deg_G(v) - s$$ \begin{equation} \implies |V_1|\log_2|V_1| + |V_2|\log_2|V_2| + 2s \geq \sum_{v \, \in \, V(G)} \deg_G(v) \label{hello} \end{equation}
However, we also know that using $\log_2$ arithmetic and using the fact that $|V_2| \geq |V_1|$: $$(|V_1| + |V_2|)\log_2(|V_1| + |V_2|) \geq |V_1|\log_2|V_1| + |V_2|\log_2|V_2| + 2|V_1|$$ thus proving the claim.
\end{proof}

\begin{lemma}
Suppose $G$ is a $2^{n-1}$-vertex induced subgraph of $Q_n$ containing edges from all the $n$ directions. Then, $\Delta(G) > \frac12\log n - \frac12 \log\log n + \frac12$.
\end{lemma}
\begin{remark}
How does proving this result imply (\ref{degGv})?
\end{remark}
\begin{proof}
Let $i \, \in \, [n]$ be an arbitrary but fixed dimension. Let us now construct three sets of vertices:
\begin{enumerate}
    \item $X_i = \{x\in V(G): x^{(i)} \in V(G)\}$, the set of endpoints of the edges of $G$ in direction $i$,
    \item $Y_i = \{y \notin V(G): y^{(i)} \notin V(G)\} \equiv \{y \in V(Q_n \setminus G): y^{(i)} \in V(Q_n \setminus G)\}$, and
    \item $A_i = V(Q_n) \setminus X_i \setminus Y_i$.
\end{enumerate}
Let $\Delta = \Delta(G)$ and consider a pair $x, x^{(i)} \in X_i$. Now:
\begin{claim}
$x$ has at least $n-2\Delta+1$ neighbours in $Y_i$.
\end{claim}
\begin{proof}
Let $y \in A_i$. Now, using the definition of $\Delta(G)$ and from the fact that either one of $y$ or $y^{(i)}$ is in $V(G)$, we have the following:
\begin{enumerate}
    \item There are at most $\Delta - 1$ neighbours of $x$ in $A_i$ (because we know $x$ and $x^{(i)}$ are adjacent). Let's denote them by $\{x^{(j)}: \#\{j\} \leq \Delta - 1\}$. 
    \item There are at most $\Delta - 1$ neighbours of $x^{(i)}$ neighbours in $A_i$, which are nothing but $\{x^{(ij)}: \#\{j\} \leq \Delta - 1\}$, where $x^{(ij)}$ agrees with $x$ in every dimension except $i$ and $j$.
    \item Consequently, there are at most $2(\Delta - 1)$ neighbours of $x$ in $A_i$.
    \item There are at least $n-1-2(\Delta - 1) = n - 2\Delta + 1$ neighbours of $x$ in $Y_i$.
\end{enumerate}
\end{proof}
We observe that $|V(G)| = 2^{n-1} \implies |V(Q_n \setminus G)| = 2^{n-1}$ and that if $G$ contains edges from all $n$ directions, then $Q_n \setminus G$ also contains edges from all the $n$ directions. Consequently, $|X_i| = |Y_i|$ and it's obvious that they must be greater than $0$.
Using these facts, we obtain: $$|E(G(X_i \cup Y_i))| \geq \frac12|X_i| + \frac12 |Y_i| + (n- 2\Delta + 1)|X_i|$$ from adding the edges within the graphs $X_i$, $Y_i$ and also between them. After some calculations and using the fact that $|X_i| = |Y_i|$, we get: $$\bar{d}(G(X_i \cup Y_i)) \geq n - 2\Delta + 2$$ which when used with Lemma 3.8 yields: $|X_i| \geq 2^{n - 2\Delta + 1}$.\newline

Thus, there are at least $2^{n-2\Delta + 1}$ edges in direction $i$ in the graph $G$. Summing over all dimensions, we observe that at least $n \cdot 2^{n - 2\Delta + 1}$ edges are in the graph $G$. On the other hand, this number cannot exceed $\Delta \cdot 2^{n-1}$. Setting $n \cdot 2^{n - 2\Delta + 1} \leq \Delta \cdot 2^{n-1}$, a straightforward computation shows that $$\Delta(G) \geq \frac12\log n - \frac12 \log\log n + \frac12$$ This proves the theorem.
\end{proof}

\section{Nisan-Szegedy's result}
\label{sec:Nisan}

\textbf{\underline{Aim:}} In this section, we will look at Nisan, Szegedy's results \cite{Nisz} that establish a polynomial relation between the decision tree complexity of Boolean functions, the degree of the multilinear polynomial representing it, and the smallest degree of the polynomial approximating the Boolean function.  

\subsection{Introduction}
The Sensitivity Conjecture was formally proposed by Nisan and Szegedy after characterising the degree of Boolean functions represented by multilinear polynomials in terms of the Boolean function's combinatorial properties. If we define a Boolean function $f$ as $f: \{F, T\}^n \rightarrow \{F, T\}$, then we can encode $T$ as $1$ and $F$ as $0$, thus creating a mapping from a subset of $\bS^n$ to a subset of $\bS$.

\begin{definition}
A real multivariate polynomial $p: \bS^n \rightarrow \bS$ \textit{represents} $f$ if for all $x \in \{0, 1\}^n$, $f(x) = p(x)$. 
\end{definition}
There is always a unique multilinear polynomial representing a given Boolean function. The results established by Nisan and Szegedy concern themselves with the most basic parameter in representing Boolean functions as polynomials, its degree.

\begin{definition}
The \textit{degree} of a Boolean function $f$, denoted by $\deg(f)$, as the degree of the unique multilinear real polynomial that represents $f$ exactly.
\end{definition}

The first result we would like to show is establishing the lower bound on the degree of a Boolean function $f$ in terms of the number of variables $n$.
\begin{theorem} \label{thm:neat}
Let $f$ be a Boolean function that depends on $n$ variables. Then, \begin{equation} \deg(f) \geq \log_2 n - O(\log \log n) \label{degree}
\end{equation}
\end{theorem}

We next show a relation between the degree of the Boolean function $f$ and the decision tree complexity of $f$ along with the smallest degree of a polynomial that approximates $f$, denoted by $\widetilde{\deg}(f)$. 
\begin{definition}
The polynomial $p$ \textit{approximates} $f$ if for every $x \in \{0, 1\}^n$, we have that \begin{equation} |p(x) - f(x)| < \frac13 \label{approx} \end{equation}
\end{definition}
The approximate degree of $f$, that is, $\widetilde{\deg}(f)$ is defined to be the minimum over all polynomials $p$ that approximate $f$ of the degree $p$. Hence, we have the following inequalities:
\begin{theorem} For every Boolean function $f$, \begin{equation}  \deg(f) \leq D(f) \leq 16\deg(f)^8 \label{degD} \end{equation} \end{theorem}
\begin{theorem} There exists a constant $c$ such that for every Boolean function, we have: \begin{equation} \widetilde{\deg}(f) \leq \deg(f) \leq D(f) \leq c\cdot \widetilde{\deg}(f)^8 \label{power} \end{equation} \end{theorem}

\subsection{Some Fourier Analysis}
In this subsection, we state and prove some results needed for our discussions from \cite{BoolFun}. Consider the Fourier transform representation, that is, considering the Boolean function $f$ as $f:\{-1, 1\}^n \rightarrow \{-1, 1\}$. Interpret the domain $\{-1, 1\}^n$ of $f$ as $2^n$ points lying in $\bS^n$ and think of $f$ as giving a $\pm 1$ labeling to each of these points. There is a familiar method for interpolating such data points with a polynomial. We'll let the reader refer to \cite{BoolFun} for more details.\newline

Generally, we observe that this polynomial is always "multilinear" - that is, it has no variables squared, cubed, etc. In general, a multilinear polynomial over variables $x_1, \cdots, x_n$ has $2^n$ terms, one for each monomial $\prod_{i \in S} x_i$, where $S \subseteq [n]$ (Note: $\prod_{i \in \phi} x_i$ denotes 1). Hence: 
\begin{lemma}
Every function $f: \{-1, 1\}^n \rightarrow \bS$ can be expressed uniquely as a multilinear polynomial, \begin{equation}
    f(x) = \sum_{S \, \subseteq \, [n]} c_S \prod_{i \in S} x_i \label{Fourier}
\end{equation} where each $c_S$ is a real number.
\end{lemma}
This expression (\ref{Fourier}) is precisely the "Fourier expansion" of $f$. It's a convention to write the coefficient $c_S$ as $\hat{f}(S)$ and the monomial $\prod_{i \in S} x_i$ as $\chi_S(x)$. Thus we finally have: \begin{equation}
    f(x) = \sum_{S \, \subseteq \, [n]} \hat{f}(S) \chi_S(x) \label{ConFour}
\end{equation}

Let us denote $\mathbf{x} = (\mathbf{x_1}, \cdots, \mathbf{x_n})$ to denote a uniformly random string from $\{-1, 1\}^n$ where each $\mathbf{x_i}$ is a random variable. We can think of generating such an $\mathbf{x}$ by choosing each bit $\mathbf{x_i}$ independently and uniformly from $\{-1, 1\}$. The result we'll need for our discussions is:
\begin{theorem}
(Parseval) For any $f: \{-1, 1\}^n \rightarrow \bS$, $$ \sum_{S \, \subseteq \, [n]} \hat{f}(S)^2 = \mathbf{E}_{\mathbf{x}}[f(\mathbf{x})^2]$$
\end{theorem}
\begin{proof}
By the Fourier expansion of $f$, \begin{align*} \mathbf{E}_{\mathbf{x}}[f(\mathbf{x})^2] &= \mathbf{E}_{\mathbf{x}}\Bigg[\Bigg(\sum_{S \, \subseteq \, [n]} \hat{f}(S)\chi_S(\mathbf{x})\Bigg)^2 \Bigg] \\ &=\mathbf{E}_{\mathbf{x}}\bigg[\sum_{S, T \, \subseteq \, [n]} \hat{f}(S)\hat{f}(T)\chi_S(\mathbf{x})\chi_T(\mathbf{x})\bigg] \\&=\sum_{S, T \, \subseteq \, [n]} \hat{f}(S)\hat{f}(T)\mathbf{E}_{\mathbf{x}}[\chi_S(\mathbf{x})\chi_T(\mathbf{x})]
\end{align*}

Recalling that $\chi_S(x)$ denotes $\prod_{i \in S} x_i$, we see that $\chi_S(x)\chi_T(x) = \chi_{S\, \Delta T}(x)$. This is because whenever $i \in S \cap T$, we get an $x_i^2$, which can be replaced by $1$. So, we get \begin{align*}
    \sum_{S, T \, \subseteq \, [n]} \hat{f}(S)\hat{f}(T)\mathbf{E}_{\mathbf{x}}[\chi_S(\mathbf{x})\chi_T(\mathbf{x})] = \sum_{S, T \, \subseteq \, [n]} \hat{f}(S)\hat{f}(T) \mathbf{E}_{\mathbf{x}} [\chi_{S \, \Delta T}(\mathbf{x})]
\end{align*}
We now observe that $\mathbf{E}_{\mathbf{x}}[\chi_U(\mathbf{x})] = 0$, unless $U = \phi$ in which case we get a $1$. This holds because by independence of the random bits $\mathbf{x_1}, \mathbf{x_2}, \cdots, \mathbf{x_n}$, we have $\mathbf{E}_{\mathbf{x}}[\chi_U(\mathbf{x})] = \mathbf{E}_{\mathbf{x}}[\prod_{i \in x_i} \mathbf{x_i}] = \prod_{i \in U} \mathbf{E}_{\mathbf{x}}[\mathbf{x_i}]$ and each $\mathbf{E}[\mathbf{x_i}] = 0$. \newline

Finally, we deduce that \begin{align*}
    \sum_{S, T \, \subseteq \, [n]} \hat{f}(S)\hat{f}(T) \mathbf{E}_{\mathbf{x}} [\chi_{S \, \Delta T}(\mathbf{x})] = \sum_{S \Delta T = \phi}\hat{f}(S)\hat{f}(T) = \sum_{S \, \subseteq \, [n]} \hat{f}(S)^2
\end{align*} as claimed.
\end{proof}


Finally for Boolean functions $f: \{-1, 1\}^n \rightarrow \{-1, 1\}$, we have $f(x)^2 = 1$ for every $x$, hence:
\begin{corollary}
(Parseval's equality) If $f: \{-1, 1\}^n \rightarrow \{-1, 1\}$, then $$\sum_{S \, \subseteq \, [n]} \hat{f}(S)^2 = 1$$
\end{corollary}

\subsection{Proving Theorem 4.3}
We will make use of an important definition due to Kalai \cite{KKL}:
\begin{definition}
For a Boolean function on $n$ variables and a variable $x_i$, the \textit{influence} of $x_i$ on $f$, denoted by $\text{Inf}_{i}(f)$, is defined to be: $$\text{Inf}_i(f) = \Pr[f(x) \neq f(x^{(i)})]$$  where $x^{(i)}$ denotes the string $x \in \{0, 1\}^n$ with the $i$th bit flipped and $x$ is chosen uniformly in $\{\text{false}, \text{true}\}^n$. 
\end{definition}
In words, $\text{Inf}_i(f)$ is the probability that flipping the $i$th coordinate flips the value of the function. After some computation, we observe that:
\begin{lemma}
For any Boolean function $f$ on $n$ variables, using the Fourier transform representation of $f$, we have: \begin{equation}
    \sum_{i=1}^{n} \text{Inf}_{i}(f) = \sum_{S} |S|\hat{f}(S)^2 \label{Infour}
\end{equation}
\end{lemma}
Combining this lemma with Parseval's equality, we can conclude that:
\begin{corollary}
For any Boolean function $f$, \begin{equation} \sum_{i=1}^{n} \text{Inf}_{i}(f) \leq \deg(f) \label{Infdeg} \end{equation}
\end{corollary}

Now, we need the Schwartz-Zippel lemma that gives an upper bound for the number of $\{-1, 1\}$ zeroes of any multilinear polynomial.
\begin{lemma}
Let $p(x_1, \cdots, x_n)$ be a non-zero multilinear polynomial of degree $d$. If we choose $x_1, \cdots, x_n$ at random in $\{-1, 1\}$, then:
$$\Pr[p(x_1, \cdots, x_n) \neq 0] \geq 2^{-d} \label{probs}$$\qed
\end{lemma}
We will leave it to the reader to prove this lemma by an induction on $n$ and writing $p$ as a linear combination of two polynomials.

\begin{proof}[Proof of Theorem \ref{thm:neat}]
For each $i$ define a function $f^i$ on $n-1$ variables as follows: $$f^{i}(x_1, \cdots, x_{i-1}, x_{i+1}, \cdots, x_n) = f(x_1, \cdots, -1, x_{i+1}, \cdots, x_n) - f(x_1, \cdots, 1, x_{i+1}, \cdots, x_n)$$
 
Now from the definition of influence we  easily get that: $$\text{Inf}_{i}(f) = \Pr[f^{i}(x_1, \cdots, x_{i-1}, x_{i+1}, \cdots, x_n) \neq 0]$$ where $x_1, \cdots, x_{i_1}, x_{i+1}, \cdots, x_n$ are chosen at random in $\{-1, 1\}$.\newline
 
Since $f$ depends on all the variables, we have that for every $i$, $f^{i}$ is not identically zero, and thus we can use Lemma \ref{probs} and conclude that for all $i$, $\text{Inf}_{i}(f) \geq 2^{-d}$. On the other hand, from (\ref{Infdeg}) it follows that $\sum_{i} \text{Inf}_{i}(f) \leq d$. Combining these two bounds gives us the required result.
\end{proof}

\subsection{Proving Theorems 4.5 and 4.6}
Now, we will return to the representation of $\text{true}=1$ and $\text{false}=0$. Also, we'll need the notion of symmetric polynomials. 

\begin{definition}
A polynomial $f(x_1, x_2, \cdots, x_n)$ is \textit{symmetric} if $$f(x_1, \cdots, x_n) = f(x_{\sigma(1)}, \cdots, x_{\sigma(n)})$$ for any permutation $\sigma$ of $[n] = \{1, \cdots, n\}$.
\end{definition}We will next use the method of symmetrization. Let $p: \bS^n \rightarrow \bS$ be a multivariate polynomial.

\begin{definition}
The \textit{symmetrization} of $p$ is $$p^{\text{sym}}(x_1, \cdots, x_n) = \frac{\sum_{\sigma \in S_n} p(x_{\sigma(1)}, \cdots, x_{\sigma(n)})}{n!}$$ 
\end{definition}
Since, the number of permutations of $[n]$ is $n!$, we get a sum of $n!$ terms each term resolving to give $0$ or $1$. 

\begin{example}
Let us consider a polynomial $p:\{0, 1\}^2 \rightarrow \bS$ given by $p(x_1, x_2) = x_1 + x_1x_2$. By the laws of Boolean algebra we know that: $x_1 + x_1x_2 = x_1(1 + x_2) = x_1$. Then, we get: $$p^{\text{sym}}(x_1, x_2) = \frac{x_1 + x_2}{2}$$
\end{example}

Hence, the important point to note here is that if we are only interested in inputs $x \in \{0, 1\}^n$ then $p^{\text{sym}}$ turns out to depend only upon $x_1 + \cdots + x_n$. We can thus represent it as a univariate polynomial of $x_1 + \cdots + x_n$ (see \cite{percep}):
\begin{lemma}
Taking $p:\bS^n \rightarrow \bS$ to be a multivariate polynomial, then there exists a unique univariate polynomial $\tilde{p}: \bS \rightarrow \bS$ such that for all $x_1, \cdots , x_n \in \{0, 1\}^n$, we have: $$p^{\text{sym}}(x_1, \cdots, x_n) = \tilde{p}(x_1 + \cdots + x_n)$$ Moreover, $\deg(\tilde{p}) \leq \deg(p)$.
\end{lemma}
\begin{proof}
Let the degree of $p^{\text{sym}}$ be $d$. Let $P_k$ denote the sum of all $\binom{n}{k}$ products $\prod_{i \, \in \, S} x_i$ of the $|S| = k$ different variables. Since $p^{\text{sym}}$ is symmetric, it can be shown by induction that it can be written as: $$p^{\text{sym}}(x) = c_o + c_1P_1(x) + c_2P_2(x) + \cdots + c_dP_d(x)$$ with $c_i \in \bS$. Observe that on $x \in \{0, 1\}^n$ with $z: = x_1 + \cdots + x_n$ ones, $P_k$ assumes value: $$P_k(x) = \binom{z}{k} = \frac{z(z-1)\cdots(z-k+1)}{k!}$$ which is a polynomial of degree $k$ of $z$. Therefore the univariate polynomial $\tilde{p}(z)$ defined by: $$\tilde{p}(z) := c_0 + c_1 \binom{z}{1} + c_2 \binom{z}{2} + \cdots + c_d \binom{z}{d}$$ has the desired property.
\end{proof}

As $\deg(p)$ is bounded below by $\deg(\tilde{p})$, we wish to find a bound for $\deg(\tilde{p})$. We will therefore need to use a result from approximation theory (see, e.g. \cite{nia}).
\begin{theorem}
(Markov Inequality) Let $p: \bS \rightarrow \bS$ be a univariate polynomial of degree $d$ that for any real number $a_1 \leq x \leq a_2$ satisfies $b_1 \leq p(x) \leq b_2$. Then, for all $a_1 \leq x \leq a_2$, the derivative of $p$ satisfies $$|p'(x)| \leq \frac{d^2(b_2 - b_1)}{a_2 - a_1}$$
\end{theorem}

In our case, because we have information on the values of $p(x)$ only for integer $x$, we need this lemma next:

\begin{lemma}
Let $p$ be a polynomial with the following properties:\begin{enumerate}
    \item For any integer $0 \leq i \leq n$, we have $b_1 \leq p(i) \leq b_2$.
    \item For some real $0 \leq x \leq n$, the derivative of $p$ satisfies $|p'(x)| \geq c$.
\end{enumerate}
Then $$\deg(p) \geq \sqrt{\frac{cn}{c + b_2 - b_1}}$$
\end{lemma}
\begin{proof}
Let $c' = \max_{0 \leq x \leq n} |p'(x)|$ which is definitely greater than or equal to $c$. It is also clear that for all real $0 \leq x \leq n$: $$b_1 - \frac{c'}{2} \leq p(x) \leq b_2 + \frac{c'}{2}$$ Using the Markov inequality, we have: $$c' \leq \frac{\deg(p)^2(c' + b_2 - b_1)}{n}$$ Thus, $$\deg(p)^2 \geq \frac{c'n}{c' + b_2 - b_1} \geq \frac{cn}{c + b_2 - b_1}$$
\end{proof}

\begin{lemma}
Let $f$ be a Boolean function such that $f(000 \cdots 0) = 0$ and for every Boolean vector $\mathbf{x}$ of Hamming weight $1$, $f(\mathbf{x}) = 1$. Then: \begin{equation}
    \deg(f) \geq \sqrt{\frac{n}{2}} \label{degtwo}
\end{equation} and \begin{equation}
    \widetilde{\deg}(f) \geq \sqrt{\frac{n}{6}} \label{degsix}
\end{equation}
\end{lemma}
\begin{proof}
We will first prove the bound for $\widetilde{\deg}(f)$. The sharper bound for $\deg(f)$ follows exactly the same lines. Let $p$ be a univariate polynomial approximating $f$, and consider $\tilde{p}$ the univariate polynomial giving its symmetrization. Now, $\tilde{p}$ satisfies the following properties:\begin{enumerate}
    \item By Lemma 4.17, we have $\deg(\tilde{p}) \leq \deg(p)$.
    \item Since for every Boolean vector $\mathbf{x}$, then $p(\mathbf{x})$ is within $\frac13$ of a Boolean value (by equation $(8)$), so for every integer $0 \leq i \leq n, -\frac13 \leq \tilde{p}(i) \leq \frac43$.
    \item Since, $f(000 \cdots 0) = 0$, we get $\tilde{p}(0) \leq \frac13$ by equation $(8)$.
    \item Since for all Boolean vectors $\mathbf{x}$ of Hamming weight $1$, $f(\mathbf{x}) = 1$ using equation $(8)$ once again, we get $\tilde{p}(1) \geq \frac23$.
\end{enumerate}

By the properties $(3)$ and $(4)$ listed above and on using the Mean Value Theorem (MVT) for derivatives, we obtain for some real $0 \leq z \leq 1$, the derivative $\tilde{p}'(z) \geq \frac{1}{3}$. We can now apply Lemma 4.19 to obtain the lower bound for $\deg(\tilde{p})$. We can also get a similar bound for $\deg(p)$ by using $0 \leq \tilde{p}(i) \leq 1$ along with Lemma 4.19 and the MVT for derivatives.
\end{proof}

The authors also give an example of a function $f$ satisfying $f(\mathbf{0})=0$ and $f(e_i) = 1$ where $\mathbf{0}$ is the zero vector and $e_i$ is the Boolean vector with one $1$ at the $i$-th position whose degree is significantly smaller than $n$.
\begin{lemma}
There exists an (explicitly given) Boolean function $f$ of $n$ variables satisfying $f(\mathbf{0})=0$ and $f(e_i) = 1$ and $\deg(f) = n^{\alpha}$ for $\alpha = \log_{3}2 = 0.631\cdots$.\qed
\end{lemma}
We leave it to the reader to prove the lemma by constructing a Boolean function $E_3(x, y, z)$ and building a complete ternary tree on the $n$ variables by considering an extension of the function. \newline 

Now, Nisan \cite{earlier} had previously proved an polynomial relation between the block sensitivity and the decision tree complexity as: 
\begin{equation}
    bs(f) \leq D(f) \leq bs^4(f) \label{bsD}
\end{equation}

\begin{lemma}
For every Boolean function $f$, \begin{equation}
    \deg(f) \geq \sqrt{\frac{bs(f)}{2}} \label{degbs}
\end{equation} and \begin{equation}
    \widetilde{\deg}(f) \geq \sqrt{\frac{bs(f)}{6}}
\end{equation}
\end{lemma}
\begin{proof}
We refer to Theorem 14.11, \cite{stasys} for the proof of this lemma. Let $f(x)$ be a Boolean function on $n$ variables, and let $q: \bS^n \rightarrow \bS$ be the multilinear polynomial of degree $d$ representing $f$. By Lemma 4.20, we know that every Boolean function $f$ of $n$ variables, which rejects the zero vector and accepts all $n$ vectors with Hamming weight $1$, has $\deg(f) \geq \sqrt{\frac{n}2}$. It's therefore enough to construct a multilinear polynomial $p$ of $t = bs(f)$ variables satisfying the conditions of this lemma.\newline

Let $t = bs(f)$, and $a$ and $S_1, \cdots, S_t$ be the input and the sets achieving the block sensitivity. Let us assume WLOG that $f(a) = 0$. We transform $q(x_1, x_2, \cdots, x_n)$ into a multilinear polynomial $p(y_1, y_2, \cdots, y_t)$ of $t$ new variables by replacing every variable $x_j$ in $p$ as follows:
Define a function $f'(y_1, \cdots, y_t)$ as follows:
$$x_j: = \begin{cases}
y_i, &\text{if } a_j = 0 \, \, \text{and } j \in S_i\\
1 - y_i, &\text{if } a_j = 1 \, \, \text{and } j \in S_i\\
a_j & \text{if } j\notin S_1 \cup \cdots \cup S_t
\end{cases}$$ That is, for $y \in \{0, 1\}^t$ we have that:
$$p(y) = q(a \oplus y_1S_1 \oplus \cdots \oplus y_tS_t)$$ where $$y_iS_i = (0, \cdots, 0, \overbrace{y_1, \cdots, y_i}^{S_i}, 0, \cdots, 0)$$ We can easily verify that $p$ is a multilinear polynomial of degree at most $d$, and satisfies the conditions of Lemma 4.20. We can therefore conclude that $$d = \deg(q) \geq \deg(p) \geq \sqrt{\frac{t}2} = \sqrt{\frac{bs(f)}2}$$ The proof of the other part is analogous.
\end{proof}
Since $\deg(f) \leq D(f)$ is obvious (How? Note that the tests along paths to $1$-leaves define a multilinear polynomial), by combining (\ref{bsD}) and the results of Lemma 4.20, we have (\ref{degD}) and (\ref{power}).

\section{Gotsman-Linial's result}
\label{sec:GL}

\textbf{\underline{Aim:}} In this section, we look at Gotsman and Linial's result \cite{Gotsman} that reduces the Sensitivity Conjecture to answering a 'simple' question about cubes of different dimensions: If you choose any collection of more than half the corners of a cube and color them red, is there always some red point that is connected to many other red points? 

\subsection{The Theorem}

Let us consider a Boolean function $f: \{-1, 1\}^n \rightarrow \{-1, 1\}$. Let $g$ be the same function as $f$ except that we flip the value on all odd vertices. Notice now that the sensitivity of $f$ on $x$ is the number of $i$ such that $$g(x_1, \cdots, x_i, \cdots, x_n) = g(x_1, \cdots, -x_i, \cdots, x_n)$$

This can be easily visualized as follows: Take a Boolean function $f$ and for any $x \in Q_n$, calculate $f(x)$, marking the vertex blue if $f(x) = 1$ and red, if otherwise. Thus, the sensitivity of $f$ on $x$ using $f$ is defined to be the number of neighbours of $x$ in $Q_n$ that have a different color than $x$. We can also easily define sensitivity using $g$.

Let now G be the induced subgraph of vertices of $x$ such that $g(x) = -1$ if the node $x$ was colored blue and H be the induced subgraph on the set of $x$ such that $g(x) = 1$ if the node $x$ was colored red. The sensitivity of $f$, using this new notation, is defined to be the maximum number of neighbours of any vertex in $G$ or $H$. Now, consider $f$ as a multilinear polynomial over the reals which is possible when we consider $f$ as a function going from $\bS^n$ to $\bS$. Now, the Sensitivity Conjecture states that there is some $\alpha > 0$ such that if $f$ has degree $n$, then $f$ has sensitivity at least $n^\alpha$.\newline


Now, we've seen earlier how Nisan and Szegedy show that the degree $d(f)$ is equivalent to the other complexity measures of Boolean functions such as: block sensitivity, certificate complexity and decision tree depth. Szegedy also independently proved a relation between $d(f)$ and $s(f)$ as: $$d(f) \geq \sqrt{s(f)}$$

Let's denote $\Gamma(G) = \max(\Delta(G), \Delta(Q_n \setminus G))$. Gotsman and Linial proved the following remarkable equivalence that helps in proposing an upper bound for $d(f)$ in terms of $s(f)$:
\begin{theorem} The following are equivalent for any function $h :\bN \rightarrow \bS$:
\begin{enumerate}
\item For any induced sub-graph G of $Q_n$ such that $|V(G)| \neq 2^{n-1},\, \Gamma(G) \geq h(n).$
\item For any Boolean function $f$, $h(d(f)) < s(f).$\end{enumerate}

\begin{remark}
In graph-theoretical terms: Suppose you have a partition of the hypercube graph $Q_n$ into sets $A$ and $B$ such that $|A| \neq  |B|$, and let $G$ and $H$ be the induced sub-graphs of $A$ and $B$. Then, there is some constant $\alpha > 0$ such that there is a node of $A$ or $B$ with at least $n^\alpha$ neighbours.
\end{remark}

\begin{remark}
Hao proved that given any subset $A$ of the vertices of a hypercube with $|A|>2^{n-1}$, the induced subgraph has a node of degree at least $\sqrt{n}$. Since either A or B in the G-L assumption has size greater than $2^{n-1}$, Hao's result proves the Sensitivity conjecture.
\end{remark}

\begin{proof}
Let's convert statement $1$ of the theorem into a statement on Boolean functions: Associate with the induced subgraph $G$ a Boolean function $g$ such that $g(x) =1$ iff $x \in V(G)$. How can we do this? One way would be to use Karnaugh maps. Now observe that $\deg_{G}(x) = n - s(g, x)$ for $x \in V(G)$ and the same relation also holds in $Q_n \setminus G$ for $x \notin V(G)$. (How?)\newline

Let's denote by $\mathbf{E}(g) = 2^{-n}\sum_{x}g(x)$ the average value of $g$ on $Q_n$. Now, statements $1$ and $2$ of the theorem can be seen as equivalent to:
\begin{lemma}  Theorem 5.1 can be reformulated as:

\begin{enumerate}[I]
\item For any Boolean function $g$, $\mathbf{E}(g) \neq 0 \implies \exists x : s(g, x) \leq n - h(n)$. 
\item For any Boolean function $f$, $s(f) < h(n) \implies d(f) < n$. \end{enumerate}

\begin{proof} 
\underline{Seeing $1 \rightarrow I$}: \begin{align*} \Gamma(G) = \max(\Delta(G), \Delta(Q_n \setminus G)) \geq h(n) &\implies \Delta(G) \geq h(n)\\ &\implies \max_{x \in V(G)} \deg_{G}(x) \geq h(n) \\ &\implies \exists x: \deg_{G}(x) \geq h(n) \\ &\implies \exists x: n-s(g, x) \geq h(n) \\ &\implies \exists x: s(g, x) \leq n-h(n) \end{align*} where the first implication is done assuming that $\Delta(G) \geq \Delta(Q_n \setminus G)$.\newline
\underline{Seeing $I \rightarrow 1$}: The requirement of $\mathbf{E}(g) \neq 0$ corresponds to $|V(G)| \neq 2^{n-1}$. Now, we can just reverse the implications of the previous result to prove the equivalence of $1$ and $I$.\newline

\noindent \underline{Seeing $2 \rightarrow II$}: Given $d(f) < h^{-1}(s(f))$, then it's easy to see that if $s(f) < h(n)$, the result follows.\newline
\underline{Seeing $II \rightarrow 2$}: To prove the reverse implication, let $f$ be a Boolean function of degree $d$. Fix a monomial of degree $d$ of the representing polynomial of $f$. Without loss of generality, we may assume the monomial is $x_1\cdots x_d$. Define $g(x_1, \cdots, x_d) := f(x_1, \cdots, x_d, 0, \cdots, 0)$. Then, $g$ has full degree $d$, so it follows that $s(f) \geq s(g) \geq h(d) = h(\deg f)$, as desired.
\end{proof}
\end{lemma}

To see the equivalence of $I$ and $II$, define $$g(x) = f(x)p(x)$$ where $p(x) = (-1)^{\sum x_i}$ is the parity function (note, we take the range of Boolean functions to be $\{-1, +1\}$). Since the parity function is sensitive to all $n$ variables, we observe that $\forall x \in Q_n$: \begin{equation} s(g, x) = n - s(f, x) \label{one} \end{equation} and also for all $S \subset [n]:$ \begin{equation} \hat{g}(S) = 2^{-n}\sum_{x}g(x)\prod_{i \, \in S} x_i= 2^{-n}\sum_{x}f(x)\prod_{i\, \notin S} x_i = f([n] \setminus S) \label{two} \end{equation} In particular, \begin{equation}\mathbf{E}(g) = \hat{g}(\phi) = \hat{f}([n]) \label{three} \end{equation} where $\hat{f}(S)$ denotes the Fourier transform of $f$ at $S \subset [n]$, that is, the highest order coefficient in the representation of $f$ as a polynomial. \newline

\underline{Seeing $I \rightarrow II$}: Assume that $d(f) = n$, i.e, $\hat{f}([n])\neq 0$. By (\ref{three}), $\mathbf{E}(g) \neq 0$ and by $I$, \begin{align*} \exists x: s(g, x) \leq n - h(n) \implies \exists x: s(f, x) \geq h(n)\end{align*} a contradiction.

\underline{Seeing $II \rightarrow I$}: Assume \begin{align*}\forall x: s(g, x) = n - s(f, x) > n - h(n) \implies s(f) < h(n)\end{align*} as $s(f) = \max_{x \in Q_n} s(f, x)$. Then by $II$, \begin{align*} d(f) < n \implies \hat{f}([n]) = \hat{g}(\phi) = \mathbf{E}(g) = 0\end{align*} a contradiction.

\end{proof}
\end{theorem}

\subsection{Concluding Remarks}

Now, Gotsman-Linial's result translates a Boolean function with a polynomial gap between degree and sensitivity into a graph with the same polynomial gap between $\Gamma$ and $n$, and vice-versa. For example, observe that Rubinstein's function (given below) can be used to obtain a graph $G$ with the surprising property $\Gamma(G) = \Theta(\sqrt{n})$. As we saw previously, Chung et.al \cite{Chung} constructed a graph $G$ with $\Gamma(G) < \sqrt{n} + 1$. This theorem implies that the following conjecture is equivalent to the Sensitivity Conjecture.

\begin{lemma} There is a constant $\epsilon > 0$ such that for every induced subgraph $G$ of $Q_n$ with $|V(G)| \neq 2^{n-1}$ we have $\Gamma(G) \geq n^\epsilon$.\end{lemma}

\subsubsection{Rubinstein's function}
The following function was constructed by Rubinstein \cite{rubi}: Assume we have $n = k^2$ variables ($k$ even), which are divided into $k$ blocks of $k$ variables each. The value of the function is $1$ if there is at least one block with exactly two consecutive $1$s in it, and it is $0$ otherwise.\newline

The block sensitivity of Rubinstein's function is shown to be equal to $n/2$ (hence, the certificate complexity and the decision-tree complexity is at least $n/2$) and the sensitivity is $\sqrt{n}$; this can be verified by a direct computation of $\hat{f}([n])$. Hence, for this function, we can establish the relation: $$bs(f) \geq s(f)^2/2$$

\section{Hao's result}
\label{sec:Hao}

\textbf{\underline{Aim:}} In this section, we expound Hao's result \cite{Hao} which shows that every $(2^{n-1} + 1)$-vertex induced subgraph of the $n$-cube graph $Q_n$ has maximum degree at least $\sqrt{n}$, hence proving the Sensitivity Conjecture.

\subsection{An overview}
Rubinstein's example of quadratic discrepancy we had seen earlier is still the best known lower bound on $bs(f)$ in terms of $s(f)$. But no one had proven anything better than an exponential upper bound until Hao's result, from which it follows that: For all Boolean functions $f$, $$bs(f) \leq 2s(f)^4$$ We observe that this concrete bound is the combination of two quadratic bounds. From (\ref{degbs}), we conclude that: \begin{equation} bs(f) \leq 2\deg(f)^2\end{equation} The other is the conjecture, \begin{equation} \deg(f) \leq s(f)^2 \label{sdeg} \end{equation} in Gotsman and Linial's paper, which is what Hao proves.

\begin{theorem} For all integers $n$ greater than $1$, let $H$ be an arbitrary $(2^{n-1} + 1)$-vertex induced subgraph of $Q_n$, then $\Delta(H) \geq \sqrt{n}$. Moreover, this inequality is tight when $n$ is a perfect square. \end{theorem}

To prove this theorem, we will need to use Cauchy's interlace theorem \cite{fisk}. 

\subsection{Cauchy's Interlace Theorem}
\begin{lemma}
All eigenvalues of a Hermitian matrix are real numbers.
\end{lemma}

\begin{proof}
Let $\lambda$ be an arbitrary eigenvalue of a Hermitian matrix $A$ with corresponding eigenvector $\mathbf{x}$. Then: \begin{align*} A\mathbf{x} = \lambda\mathbf{x} \implies \mathbf{x}^TA\mathbf{x} = \mathbf{x}^T(\lambda\mathbf{x})\implies \mathbf{x}^TA\mathbf{x} = \lambda\mathbf{x}^T\mathbf{x} = \lambda||\mathbf{x}|| \end{align*}

Take the conjugate transpose on both sides to get (note $A$ is Hermitian): \begin{align*} \mathbf{x}^T\bar{A}^T\mathbf{x} = \bar{\lambda}||\mathbf{x}|| \implies \bar{\lambda}||\mathbf{x}|| = \mathbf{x}^T\bar{A}^T\mathbf{x} = \mathbf{x}^TA\mathbf{x} = \mathbf{x}^T\lambda\mathbf{x} = \lambda||\mathbf{x}||\end{align*}

Hence, $\lambda||\mathbf{x}|| = \bar{\lambda}||\mathbf{x}||$. We observe that $\mathbf{x}$ isn't the zero vector because it's an eigenvector. So, $\bar{\lambda} = \lambda \implies \lambda \in \bS$.
\end{proof}

Now, if $f$ and $g$ are polynomials with real roots $r_1 \leq r_2 \leq \cdots \leq r_n$ and $s_1 \leq s_2 \leq \cdots s_{n-1}$, we say that $f$ and $g$ \textit{interlace} if and only if $r_1 \leq s_1 \leq r_2 \leq \cdots \leq s_{n-1} \leq r_n$. We let the reader refer to \cite{Rasc} for the proof of this result.

\begin{lemma}
The roots of polynomials $f$ and $g$ interlace if and only if the linear combinations $f + \alpha g$ have all real roots for all $\alpha \in \bS$.
\end{lemma} 

\begin{corollary}
If $A$ is a Hermitian matrix, and $B$ is a principal submatrix of $A$, then the eigenvalues of $B$ interlace the eigenvalues of $A$.
\begin{proof}
Let's first define a principal submatrix.
\begin{definition}
Given a real $n \times n$ matrix $A$, a \textit{principal submatrix} of $A$ is obtained by deleting the same set of rows and columns of $A$.
\end{definition}
Now, simultaneously permuting rows and columns, if necessary, we may assume that the submatrix $B$ occupies rows $2,\cdots,n$ and columns $2,\cdots,n$, so that $A$ has the form: $$A = \begin{pmatrix} B&\mathbf{c}\\ \mathbf{c}^*&d\\ \end{pmatrix}$$ where $*$ signifies the conjugate transpose of a matrix. Choose $\alpha \in \mathbf{R}$, and consider the equation that follows from the linearity of the determinant: \begin{equation}
\begin{vmatrix} B-xI & \mathbf{c}\\ \mathbf{c}^* & d-x + \alpha\\ \end{vmatrix} = \begin{vmatrix} B-xI & \mathbf{c}\\ \mathbf{c}^* & d-x\\ \end{vmatrix} + \begin{vmatrix} B-xI & \mathbf{c}\\0 & \alpha\\ \end{vmatrix} \label{matrix} \end{equation}

Now, we can conclude from Lemma 6.2 using the RHS of (\ref{matrix}) that $|A-xI| + \alpha|B-xI|$ has all real roots for any $\alpha$, so the eigenvalues interlace.
\end{proof}
\end{corollary}

This proves the following important result:
\begin{theorem} (Cauchy's Interlace Theorem) Let $A$ be a symmetric $n \times n$ matrix, and let $B$ be a $m \times m$ submatrix of $A$, for some $m < n$. If the eigenvalues of $A$ are $ \lambda_1 \geq \lambda_2 \geq \cdots \geq \lambda_n$ and the eigenvalues of $B$ are $\mu_1 \geq \mu_2 \geq \cdots \geq \mu_m$, for all $1 \leq i \leq m$, $$ \lambda_i \geq \mu_i \geq \lambda_{i + n -m}$$ \end{theorem}

\subsection{Proving the Sensitivity Conjecture}
Our discussion of Hao's proof of the Conjecture is inspired from \cite{gllblog}. Now, it can be easily shown that the degree $d(G)$ of a graph $G$ is always at least as great as the largest eigenvalue $\lambda_1$ of the adjacency matrix $A_G$. So, let's take a $m$-vertex subgraph $H$ of $Q_n$ and let its top eigenvalue be $\lambda_1(H)$. Now, we know the entries of the adjacency matrix $A_H$ are either $0$ or $1$. Hao's insight was to realise that if we flipped the signs of some $1$s in $A_H$, thus creating a so-called "pseudo-adjacency matrix", the relation between the degree and the largest eigenvalue of the adjacency matrix still holds.
\begin{lemma}
Suppose $H$ is a $m$-vertex undirected graph, and $A$ is a symmetric matrix with entries in $\{0, \pm 1\}$ and whose rows and columns are indexed by $V(H)$, the vertex set of $H$. Also, let $A_{u, v} = 0$ whenever $u, v \in V(H)$ are non-adjacent. Then, $$\Delta(H) \geq \lambda_1 := \lambda_1(A)$$
\end{lemma}

\begin{proof}
Let $\mathbf{v}$ be the eigenvector corresponding to $\lambda_1$. Then, $\lambda_1\mathbf{v} = A\mathbf{v}$. We choose an index $i$ that maximises the absolute value $|v_i|$. Then: \begin{align*} |\lambda_1v_i| = |(A\mathbf{v})_i| = |\sum_{j} A_{i, j}v_j| \leq |\sum_{j} A_{i, j}| \cdot |v_i| \leq \sum_{(i, j) \in E(G)} |A_{i, j}| \cdot |v_i| \leq d(H)|v_i| \end{align*} from which the result follows.
\end{proof}

Why is this lemma important? Observe that a key idea to this lemma was the introduction of $-1$ entries in the adjacency matrix $A_H$. We would like to know what happens if we try to use the original (unmodified) adjacency matrix in the lemma. One readily observes that though this matrix too has its top eigenvalue at least $\sqrt{n}$, the interlacing bound is too lossy to prove this without modification. Take for example, the 3-dimensional cube $Q_3$, in which case, the eigenvalues of the unmodified adjacency matrix are $\{-3, -1, -1, -1, 1, 1, 1, 3\}$, giving the trivial bound $\lambda_1 \geq 0$. In general, the eigenvalues of the adjacency matrix of $Q_n$ will be the integers $-n, -n+2, -n+4, \cdots, n$. \newline

Having proved this lemma, we want to find conditions that force $\lambda_1 = \sqrt{n}$, particularly when $m \geq \frac{N}{2} + 1$ with $N = |V(Q_n)| = 2^n$. Hao does this by making $A$ sit inside a matrix $A_N$ (whose construction will be discussed in section 6.4) with at least $\frac{N}{2}$ eigenvalues with value $\sqrt{n}$.\newline

For realising this, we construct $A_{N-1}$ as a principal submatrix of $A_N$ by deleting its last row and column. Since, $A_N$ and $A_{N-1}$ are both real and symmetric matrices, they have real eigenvalues and so we can order them as $\lambda_1 , \cdots, \lambda_N$ and $\mu_1, \cdots, \mu_{N-1}$ in non-increasing order. From the property outlined before Lemma 6.3, we can see that their eigenvalues always interlace.\newline

Now, observe that this process can be repeated, that is, we can construct $A_{N-2}$ by knocking off another row and its corresponding column; denoting the eigenvalues of $A_{N-2}$ by $\nu_i$, we can easily see that: \begin{equation} \mu_1 \geq \nu_1 \geq \mu_2 \geq \nu_2 \geq \mu_3 \cdots \implies \lambda_1 \geq \nu_1 \geq \lambda_3 \end{equation} Repeating this procedure gives us another matrix whose top eigenvalue is still at least as big as $\lambda_4$.\newline

Generalising this for our purposes, if we repeat this process $\frac{N}{2}-1 = 2^{n-1} - 1$ times inside $A_N$, the resulting matrix's largest eigenvalue is still at least as large as $\lambda_{\frac{N}{2}}(A_N)$, which we assign the value of $\sqrt{n}$. This means the size of the resultant (square) matrix will be $m = \frac{N}{2} + 1 = 2^{n-1} + 1$ and we conclude that $$\lambda_1(A_N) \geq \lambda_1(A_m) \geq \lambda_{\frac{N}{2}}(A_N) = \sqrt{n}$$ which when used with Lemma 6.7 yields: \boxed{d(H) \geq \sqrt{n}} which is nothing but Hao's result.

\subsection{Constructing the Matrix}
Now, as we've seen earlier, Hao showed that there exists a $N \times N$ matrix $A_N$ with entries in $\{0, \pm 1\}$ whose non-zero entries correspond to the edges of the Boolean cube and such that all the $N$ eigenvalues of $A_N$ are $\pm \sqrt{n}$ and these eigenvalues sum up to zero, meaning that the trace is zero. This would then imply that $A_N^2 = nI$.\newline

Also as we saw earlier, we would ideally want $A_N$ to be the adjacency matrix of the $n$-cube but that doesn't work: that's because each $(i, j)$ entry of the square of the unmodified induced adjacency matrix of $Q_n$ counts all paths of length $2$ from node $i$ to node $j$ and that number is nonzero.\newline

This is where the introduction of $-1$ on edges comes in handy. We arrange that every $4$-cycle of the $n$-cube has exactly one edge with $-1$. Then, the pairs of paths from one corner to the opposite corner will always cancel, leaving $A_{i,j}^2 = 0$ whenever $i \neq j$. And, $A_{i, i}^2 = n$ because there are $n$ ways to go out and come back along the same edge, always contributing $1\cdot 1$ or $(-1) \cdot (-1)$ either way.\newline

Subsequently, Hao defines the needed labelling exactly by the recursion: $$A_2 = \begin{bmatrix} 0 & 1\\ 1 & 0 \end{bmatrix} \, \text{and} \,\,\, A_N = \begin{bmatrix} A_{\frac{N}{2}} & I\\ I & -A_{\frac{N}{2}}\end{bmatrix} \, \, \text{for}\, \, (N > 2)$$

We can easily verify that $A_N^2 = nI$ by induction. A nice exposition of the physical interpretation of this pseudo-adjacency matrix using physics concepts such as the Jordan-Wigner transformation and Majorana fermions is given in \cite{Yixi}.

\section{Further discussion}
\label{sec:Fur}

\textbf{\underline{Aim:}} To provide comments on the open problems posed by Hao \cite{Hao}.\newline

We refer to \cite{rgray} for the definitions in this section concerning the symmetry of graphs. 



\begin{definition}
An \textbf{automorphism} of $G$ is a bijection $f: V(G) \rightarrow V(G)$ sending edges to edges and non-edges to non-edges. 
\end{definition}

We write $H = \text{Aut}\, G$ for the full automorphism group of $G$. Roughly speaking, the more symmetry a graph has, the larger its automorphism group will be and vice-versa. 


\begin{definition}
The graph $G$ is \textit{vertex-transitive} if $H$ acts transitively on $V(G)$, that is, for all $u, v \in V(G)$, there is an automorphism $f \in H$ such that $f(u) = v$. 
\end{definition}
Intuitively, a graph is vertex-transitive if there is no structural (i.e., non-labeling) way to distinguish vertices of the graph. Some examples are the complete graph $K_n$, and the cycle $C_n$ on $n$ vertices. For a not-so-obvious example of a vertex-transitive graph, let $G$ be a group and $S \subset G$ be a generating set for $G$ such that $1_G \notin S$ and $S$ is closed under taking inverses. 

\begin{definition}
The (right) \textbf{Cayley graph} $\Gamma = \Gamma(G, S)$ is given by: \begin{align}
    V(\Gamma) = G; \, \, \, \, \, \, E(\Gamma) = \{\{g, h\}: g^{-1}h \in S\}
\end{align}
\end{definition}
Thus, two vertices are adjacent if they differ in $G$ by right multiplication by a generator. The Cayley graph of a graph is always vertex-transitive.

\subsection{Three notions of symmetry}
\subsubsection{Distance-transitive graphs}
In a connected graph $G$, we define the \textit{distance} $d(u, v)$ between $u, v \in V(G)$ to be the length of the shortest path from $u$ to $v$. 

\begin{definition}
A graph is \textit{distance-transitive} if for any two pairs of vertices $(u, v)$ and $(u', v')$ with $d(u, v) = d(u', v')$, there is an automorphism taking $u$ to $u'$ and $v$ to $v'$. 
\end{definition}
It is seen that distance-transitivity implies vertex-transitivity (How?). The Hamming graph $H(n, k)$ is defined by the vertex set $$\mathbf{Z}_k^n = \underbrace{\mathbf{Z}_k \times \cdots \mathbf{Z}_k}_{n \, \text{times}}$$ where $\mathbf{Z}_k = [k-1] = \{0, 1, \cdots, k-2, k-1\}$ and two vertices $u$ and $v$ are adjacent iff they differ in exactly one coordinate. The $d$-dimensional hypercube is defined to be $Q_d := H(n, 2)$. It's observed that Hamming graphs are a family of distance-transitive graphs (see Chapter 5, \cite{RAB}). \newline

We recall that a tree is a connected, acyclic graph and use $T_r$ to denote a regular tree where all vertices have degree $r$. A regular tree $T_r\, (r \in \mathbb{N})$ is an example of an infinite locally-finite (all vertices have finite degree) distance-transitive graph. Macpherson's theorem \cite{Mac} gives a necessary and sufficient condition for a locally-finite infinite graph to be distance transitive.

\subsubsection{Homogeneous graphs}
\begin{definition}
A graph $G$ is called \textit{homogeneous} if any isomorphism between finite induced subgraphs extends to an automorphism of the graph. 
\end{definition}
Homogeneity is the \textit{strongest} possible symmetry condition we can impose. Gardiner's result \cite{Garda} gives a concrete classification of finite homogeneous graphs and for the infinite case, we have an example of a random graph $R$ constructed by Rado \cite{Rado}.

\subsubsection{Connected-homogeneous graphs}
\begin{definition}
A graph $G$ is \textit{connected-homogeneous} if any isomorphism between connected finite induced subgraphs extends to an automorphism. 
\end{definition}
Hence, this helps us define a class of graphs between homogeneous and distance-transitive. Gardiner \cite{Gardb} gives yet another concrete classification of finite connected-homogeneous graphs. A result of Gray and Macpherson \cite{GrayMac} helps us to classify the countable connected-homogeneous graphs.\newline

With this background, we are now in a position to analyse each open problem posed by Hao.
\subsection{Problems}
\begin{problem}
Given a graph $G$ with high symmetry, what can we say about the smallest maximum degree of induced subgraphs of $G$ with $\alpha(G) + 1$ vertices, where $\alpha(G)$ denotes the size of the largest independent set in $G$?
\end{problem}
First, we need two definitions.
\begin{definition}
A set of vertices in $G$ is an \textit{independent set} if no two vertices in the set are adjacent.
\end{definition}

\begin{definition}
A \textit{maximal independent set} is an independent set to which no other vertex can be added without destroying its independence property. The number of vertices in the largest independent set of $G$ is called the \textit{independence number}, $\alpha(G)$.
\end{definition}

The result of Chung \cite{Chung} we saw earlier stating that the smallest maximum degree of $Q_n$ is at most $\lceil \sqrt{n} \, \rceil$ has been generalised in \cite{Ding} for Hamming graphs $H(n, k)$ for all $n, k \geq 1$. It has been observed that the same result holds true for these graphs also, giving a bound independent of the value of $k$.\newline

It would be interesting to find whether the classes of symmetric graphs described above would yield analogous results.

\begin{problem}
Let $g(n, k)$ be the minimum $t$ such that every $t$-vertex induced subgraph $H$ of $Q_n$ has maximum degree at least $k$. Hao showed that $g(n, \sqrt{n}) = 2^{n-1} + 1$. Can we determine $g(n, k)$ (asymptotically) for other values of $k$?
\end{problem}
We are unaware of any prior work that has considered the quantity $g(n, k)$. However, the work \cite{Geir} analyses the maximum number of vertices of degree $k$ in an induced subgraph on $n$ vertices of $Q_k$.

\begin{problem}
The best separation between the block sensitivity $bs(f)$ and the sensitivity $s(f)$ is $bs(f) = \frac23 s(f)^2 - \frac13 s(f)$, is quadratic and shown by Ambainis, Sun \cite{Amsun}. Hao proves a quartic upper bound. Is it possible to close this gap by directly applying the spectral method to Boolean functions instead of to the hypercubes?
\end{problem}
We might try to see what happens if we tweak $bs(f) = O(s(f))^4$, which was the upper bound proved by Hao, to say, $O(s(f))^2$. To check its validity, we first look at the two results used by Hao to show $bs(f) = O(s(f))^4$. The first one, shown by Hao, which essentially proved Gotsman-Linials's conjecture (see (\ref{sdeg})) was that for all Boolean functions $f:\{0, 1\}^n \rightarrow \{0, 1\}$, the inequality $s(f) \geq \sqrt{\deg(f)}$ holds. This is seen to be a tight bound for the AND-of-ORs function defined as:
\begin{definition}
The AND-of-ORs function is defined on $n$ blocks of $n$ variables each as: $$f(x_{11}, \cdots, x_{nn}) = \bigwedge_{i=1}^{n} \, \bigvee_{i=1}^{n} \, x_{ij}$$
\end{definition}
This inequality is combined with Nisan-Szegedy's result (see (\ref{degbs})) stating that for all $f:\{0, 1\}^n \rightarrow \{0, 1\}, \, \deg(f) \geq \sqrt{\frac12 bs(f)}$.  From Hatami's survey, it is known that this bound cannot be improved beyond $bs(f) \gtrsim (\deg f)^{\log_3 6}$ where $\log_3 6 \approx 1.6309$. So, the use of $\deg(f)$ isn't feasible to try to get a quadratic sensitivity upper bound on the block sensitivity.\newline

It's still open as to whether exploring the techniques of interlacing eigenvalues of signed matrices on different objects may lead to sharper bounds.

\begin{note}
It would be instructive to mention the connection of Hao's proof to Clifford algebras as shown in \cite{Dan}; this concept has been used to extend Hao's result for a Cartesian power of a directed $l$-cycle in \cite{aka}.
\end{note}

\section{Sensitivity and CREW PRAMs}
\textbf{\underline{Aim:}} To discuss how the Sensitivity Conjecture is related to CREW PRAMs.\newline

A PRAM (Parallel Random Access Machine) is the standard model for parallel computation. It consists of a set of processors $P_0, P_1, \cdots$ which communicate by means of cells $C_0, C_1, \cdots, $ of shared memory. Now, each step of computation (of a function $f$) consists of three phases: read, compute and write phases. In the read phase, each processor may choose one cell to read from. In the compute phase, an arbitrary amount of local computation can take place. In the write phase, each processor may choose one cell to write into.\newline

In the CREW (Concurrent Read Exclusive Write) PRAM variant, simultaneous read access is permitted, but not simultaneous write access. That is, several processors may read from the same location at the same time, but two or more processors may never attempt writing into the same location at the same time. \newline

Now, a key result bounding the power of ideal CREW PRAMs is by Cook, Dwork and Reischuk \cite{CDR} where they show that CREW($f$) $= \Omega(\log(s(f)))$ is a lower bound on the number of steps required to compute a function $f$ on a CREW PRAM. After this, Noam Nisan in \cite{earlier} tweaked the definition of "sensitivity" and introduced the notion of "block sensitivity" to obtain: CREW($f$) $= \Theta(\log(bs(f)))$. Now, the natural question is whether CREW($f$) $= O(\log(s(f)))$? Proving this is nothing but the Sensitivity Conjecture.




\end{document}